\documentclass[12pt, reqno]{amsart}
\usepackage{my_style}

\begin{document}
	\title{Specialisation of Soergel Bimodules}
	\author{Benjamin McDonnell}
	\thanks{I would like to thank the DFG and the Graduiertenkolleg 1821 for the financial support. This paper is based on my doctoral thesis and I thank Wolfgang Soergel for his excellent supervision.}
	\thanks{Email address: mcdonnell.benjamin@gmail.com}
	\begin{abstract}
	We consider a generalisation of the specialisation functor on Soergel bimodules and show that this generalised version still takes Soergel bimodules to Soergel modules.
	\end{abstract}
	\maketitle


	\section{Introduction}
	It is natural to specialise a Soergel bimodule at either the origin or a generic point, obtaining a Soergel module in the first case. In the second case, the bimodule splits completely according to its standard flag. 

Here, we consider subgeneric specialisations of Soergel bimodules. In this case, there is a partial splitting of the bimodules. The main result shows this decomposition consists of Soergel modules for smaller Coxeter systems, which are determined by certain cosets of the original Coxeter system.

We follow the notation of \cite{so07} and \cite{hum_cox}.
	\section{The Specialisation Theorem}

	Given a Coxeter system $(W,S)$ with length function $l$ and finite-dimensional reflection faithful representation $V$ over an algebraically closed field $\bb{K}$ with characteristic char$\bb{K} \neq 2$, we set $R = R(V)$ the $\bb{K}$-algebra of regular functions over $V$. For $x\in W$, we have the graph,
	$$\Gr(x) = \{(x\lambda, \lambda) \mid \lambda \in V\} \subset V \times V.$$
	
	For $A\subset W$, the union of graphs is denoted $\Gr(A)$, with coordinate ring $R(A)$, following the notation of \cite[3.1, 3.3]{so07}. Finally, we have the category of Soergel bimodules,
	$$\SBim.$$
			
	\subsection{An Initial Decomposition}
	
	Let $M \in \SBim$ be a Soergel bimodule. Take a point $a \in V$ with corresponding maximal ideal $\mf{a} \subset R$. We write $\bb{K}_{\mf{a}} = R/\mf{a}$ and $-\otimes_R\bb{K}_{\mf{a}}$ is written without the tensor product symbol. We would like to understand the specialisation $M\bb{K}_{\mf{a}}$. The following argument shows that $M\bb{K}_{\mf{a}}$ has a direct sum decomposition.
	
	Consider that $R\otimes \bb{K}_{\mf{a}}$ is the coordinate ring of $V \times \{a\}.$ The projection map $$pr_0: V\times\{a\} \to V$$ is an isomorphism, giving an identification of $V\times\{a\}$ with $V$.

	Now we have 
	$$\Gr(W) \cap (V \times \{a\}) = \{(wa,a)\mid w \in W\},$$
	which is identified, via $pr_0$, with $\{wa\mid w \in W\} \subset V$. Suppose the stabiliser of $a$ is $\stab(a) \subset W$. Then $$\{wa\mid w \in W\} \cong W/\stab(a).$$ 
	
	A Soergel bimodule $M\in\SBim$ is supported on $\Gr(W)$, hence 
	\begin{equation}\label{eq:support}\tag{Support}\supp(M\bb{K}_{\mf{a}}) \subset \{wa \mid w \in W\}.\end{equation}

	This immediately gives a direct sum decomposition of $M\bb{K}_{\mf{a}}$, since its support is a disjoint union of points. 
	
	\subsection{The Main Theorem}
	In fact, we can say more about this decomposition. Theorem \ref{thm:sbim_decomp} shows that the decomposition consists of Soergel modules.
	
	\begin{definition}\label{def:tits_cone}
		We have the simple system of roots $\Delta$. Let 
		$$D := \{\lambda \in V \mid (\lambda,\alpha) \geq 0 \text{ for all } \alpha \in \Delta \}$$
		Then the \textbf{Tits cone} is defined to be the union of all $w(D), w \in W$.
	\end{definition}
	\begin{lemma}\label{lem:stab_is_cox}
		Let $a \in V$ be a point in the Tits cone. Then $\stab(a)$ is generated by the reflections it contains.
	\end{lemma}
	\begin{proof}
	By \cite[5.13]{hum_cox}, $D$ is a fundamental domain for the Tits cone. The lemma then follows from a similar proof to \cite[1.12]{hum_cox}.
	\end{proof}
	Since $\stab(a)$ is generated by the reflections it contains, it must be a Coxeter group. We can upgrade this to a Coxeter system by choosing the simple reflections to be those reflections in $\stab(a)$ which have minimal length with regard to the original length function $l$ of $(W,S)$. Denote this new Coxeter system as $(\stab(a), S_a)$.
	\begin{definition}
		Let $\mf{a} \subset R$ be a maximal ideal and $a \in V$ be the point corresponding to $\mf{a}$. If $a$ is a point in the Tits cone, then call $\mf{a}$ a \textbf{Tits ideal}.
	\end{definition}
	
	\begin{theorem}\label{thm:sbim_decomp} Given a Soergel bimodule $B$ and a Tits ideal $\mf{a}$, the specialisation $B\otimes_R \bb{K}_{\mf{a}}$ decomposes as a direct sum of Soergel modules for smaller Coxeter systems.
	\end{theorem}

	The precise sense of the theorem deserves an explanation. We will see which smaller Coxeter systems are relevant and what it means to decompose as a direct sum of Soergel modules.
	
	\begin{lemma}
	Let $a \in V$ be a point in the Tits cone and $t \in W/\stab(a)$ be a minimal length representative. Then 
	$$(t\stab(a)t^{-1}, tS_at^{-1})$$ is a Coxeter system.	
	\end{lemma}
	\begin{proof}
		We already saw from lemma \ref{lem:stab_is_cox} that $(\stab(a),S_a)$ is a Coxeter system. The conjugate of a reflection is again a reflection, so $tS_at^{-1}$ is a set of reflections. Since $t$ was a minimal length representative, $tS_at^{-1}$ forms a set of simple reflections for the Coxeter group $t\stab(a)t^{-1}$.
	\end{proof}
	The previous lemma defines the Coxeter system \textit{at the point $ta$}, which we denote $(\stab(ta), S_{ta})$.
	
	The category $\SBim$ depends on a Coxeter system $(W,S)$ and a realisation $R$. We can make these choices explicit with the notation $\SBim(W,S,R)$. The precise sense of theorem \ref{thm:sbim_decomp} should be understood with the following diagram and subsequent explanation.
	$$\begin{tikzcd}
		{\SBim(W,S,R)} \arrow[r, "-\otimes_R\bb{K}_{\mf{a}}"] &[1em] R\Modall \arrow[r, "\res_{ta}"] & \SMod (t\stab(a)t^{-1},S_{ta}, R)                                          \\
		&                      & {\SBim(t\stab(a)t^{-1},S_{ta}, R)} \arrow[u, "-\otimes_R\bb{K}_{\mf{ta}}"']
	\end{tikzcd}
	$$
	where $\res_{ta}$ is the functor which projects to the direct summand supported at $ta$, as in the \eqref{eq:support} formula above. Theorem \ref{thm:sbim_decomp} claims that: for every $t \in W/\stab(a)$ a minimal length representative, the image of the horizontal composition is contained in the image of the vertical functor.
	
	\begin{remark}
		Typically, the functor from $\SBim$ to $\SMod$ (the vertical arrow in the above diagram) would be specialisation at $0$. However, we consider specialisation at $ta$ to give a functor from $\SBim$ to $\SMod$. This is essentially just an abuse of notation, but is justified because $\stab(ta) = t\stab(a)t^{-1}$. Hence specialisation at $ta$ will not cause a bimodule in $\SBim(t\stab(a)t^{-1},S_{ta}, R)$ to decompose. In this way, we see that this specialisation at $ta$ is analogous to specialisation at $0$.
	\end{remark}
	
	
	\subsection{Proof of Theorem \ref{thm:sbim_decomp}}

	\begin{notation} Let $\underline{w} = (s, t, \dots, u)$ be a sequence of simple reflections. Then write $B(\underline{w}) = R\otimes_{R^s}R\otimes_{R^t} \cdots \otimes_{R^u}R$. Bimodules of this form are called Bott-Samuelson bimodules.
	\end{notation}
	It will be sufficient to prove theorem \ref{thm:sbim_decomp} for Bott-Samuelson bimodules.
	\begin{remark}
	We can deduce the theorem for all indecomposable Soergel bimodules by an induction argument. The indecomposable bimodules are classified by $W$, and we write $B_w$ for the indecomposable corresponding to $w\in W$. Suppose the theorem holds for all indecomposables $B_x$ with $l(x)<l(w)$. Now there are integers $n_x$ such that $$B(\tu{w}) = B_w \oplus \bigoplus_{\substack{x\in W \\ l(x)<l(w)}} B_x^{\oplus n_x}$$ 
	The theorem holds for $B(\tu{w})$, as it is a Bott-Samuelson, and the theorem holds for every summand except $B_w$ by induction. It follows that the theorem holds for $B_w$ also.
	\end{remark}
	\begin{lemma}
		Given a Soergel bimodule of the form $B(\underline{w})$ for some sequence $\underline{w}$ and a Tits ideal $\mf{a}$, the specialisation $B(\underline{w})\otimes_R \bb{K}_{\mf{a}}$ decomposes as a direct sum of Bott-Samuelson modules for smaller Coxeter systems.		
	\end{lemma}
	\begin{proof}
		We proceed by induction on the length of the sequence $\underline{w}$. Suppose that the lemma holds for all sequences of length less than the length of $\tu{w}$. Suppose $\underline{w} = (s, \underline{w}')$ for some sequence $\tu{w}'$ and some simple reflection $s$.
		
		Now note that the following diagram commutes
		$$
		\begin{tikzcd}
		\SBim \arrow[d, "-\otimes_R \bb{K}_{\mf{a}}"'] \arrow[r, "B_s\otimes_R-"] & \SBim \arrow[d, "-\otimes_R \bb{K}_{\mf{a}}"] \\
		R\Modall \arrow[r, "B_s\otimes_R-"'] & R\Modall          
		\end{tikzcd}
		$$
		
		We would like to calculate $B(\tu{w})\bb{K}_{\mf{a}} =B_sB(\underline{w}')\bb{K}_{\mf{a}}$. By induction, $B(\underline{w}')\bb{K}_{\mf{a}}$ is a direct sum of Bott-Samuelson modules for smaller Coxeter systems. Let's focus our attention on the point $ta$ for some minimal length representative $t\in W/\stab(a)$. We have
		$$\res_{ta} B(\underline{w}')\bb{K}_{\mf{a}}= \bigoplus_j B(\underline{v}_j)\bb{K}_{t\mf{a}}$$
		where $\underline{v}_j$ are sequences of simple reflections in $S_{ta}$. Recall that $S_{ta}$ is the set of simple reflections of the Coxeter system ``at the point $ta$.''
		
		Now consider any chosen summand $B_sB(\underline{v}_j)\bb{K}_{t\mf{a}}$. Our goal is to show that this is a direct sum of Bott-Samuelson modules. Once we prove this, we are done.
		
		There are two cases. First, if $s \in (t\stab(a)t^{-1},S_{ta})$ then $B_sB(\underline{v}_j)$ is a Bott-Samuelson bimodule in $\SBim(t\stab(a)t^{-1},S_{ta}, R)$, and we're done.
		
		The other case is $s \notin (t\stab(a)t^{-1},S_{ta})$. Recall that  $B_s = R(e,s),$ i.e. the coordinate ring of the graph $\Gr({e,s})$. Then by \cite[3.19]{atimac},
		$$\supp B_sB(\underline{v}_j)\bb{K}_{t\mf{a}} = \{ta, sta\}.$$ 
		
		This description of the support shows that we have a splitting given by,
		$$B_sB(\underline{v}_j)\bb{K}_{t\mf{a}} = \left(R(e) \oplus R(s)\right)B(\underline{v}_j)\bb{K}_{t\mf{a}}$$
		
		To finish the proof, we refer to Lemma \ref{lem:s_twist}. We then see that 
		$$\left(R(e) \oplus R(s)\right)B(\underline{v}_j)\bb{K}_{t\mf{a}} = B(\underline{v}_j)\bb{K}_{t\mf{a}} \oplus B(s\underline{v}_js)\bb{K}_{st\mf{a}},$$
		where $s\tu{v}_js$ denotes that we conjugate every entry of the sequence $\underline{v}_j$ by $s$. Observe that both these summands are of the form required by the lemma. In fact, 
		$$B(\underline{v}_j)\bb{K}_{t\mf{a}} \in  \SMod (t\stab(a)t^{-1},S_{ta}, R)$$ and 
		$$B(s\underline{v}_js)\bb{K}_{st\mf{a}} \in \SMod (st\stab(a)t^{-1}s^{-1},S_{sta}, R).$$ 
		
		Thus we have achieved our goal and the lemma is proved.
	\end{proof}
	\begin{lemma}\label{lem:s_twist}
		Let $w \in W$ and $B(w)= R\otimes_{R^w} R$. Then we have an $(R,R)$-bimodule isomorphism $$R(s)B(w) \overset{\sim}\to B(sws)R(s).$$
		Similarly, we have an $R$-module isomorphism $R(s)\bb{K}_{w\mf{a}}\overset{\sim}\to \bb{K}_{sw\mf{a}}.$
	\end{lemma}
	\begin{proof}
		Note that there is an $(R,R)$-bimodule isomorphism given by
		$$R(s) = \frac{R \otimes R}{r\otimes 1 - 1 \otimes s(r)} \overset{\sim}\longrightarrow R$$
		$$x\otimes y \mapsto s(x)y$$ where we consider the left $R$ action on $R\in (R,R)$-Bimod to be twisted by $s$.
		
		This isomorphism allows us to write 
		$$R(s)B(w) \overset{\sim}\longrightarrow R \otimes_R R\otimes_{R^w} R \cong R \otimes_{R^w} R$$
		where the left $R$ action is twisted by $s$ and the right $R$ action is the natural one.

		Similarly,
		$$B(sws)R(s) \overset{\sim}\longrightarrow R\otimes_{R^{sws}}R\otimes_RR \cong R\otimes_{R^{sws}}R$$
		where the left $R$ action is the natural one and the right $R$ action is twisted by $s$.

		Using these two descriptions, define a map
		$$R\otimes_{R^w}R \to R\otimes_{R^{sws}}R$$
		$$r_1 \otimes r_2 \mapsto s(r_1) \otimes s(r_2).$$

		First, we check whether this is well-defined. Suppose $r \in R^w$, so that $r\otimes1=1\otimes r$. Then we need to check that $s(r)\otimes 1 = 1 \otimes s(r) \in R\otimes_{R^{sws}}R$, or in other words we should check that $s(r)\in R^{sws}$. This is clearly true, since $sws s(r) = sw(r) = s(r)$.
		
		Now we need to check that the map is indeed an $(R,R)$-bimodule homomorphism. First, check the left $R$-action:
		$$r\cdot (a\otimes b) = s(r)a\otimes b \mapsto s(s(r)a)\otimes s(b)= rs(a)\otimes s(b) = r\cdot(s(a)\otimes s(b)).$$

		Similarly, 
		$$(a\otimes b)\cdot r = a\otimes br \mapsto s(a)\otimes s(br)= s(a)\otimes s(b)s(r) = (s(a)\otimes s(b))\cdot r.$$

		Finally, it should be clear that the map is a bijection, and hence it is an $(R,R)$-bimodule isomorphism.
		
		The construction of the isomorphism $R(s)\bb{K}_{w\mf{a}}\overset{\sim}\to \bb{K}_{sw\mf{a}}$ is analogous, and in fact easier.		
	\end{proof}
	
	The proof of theorem \ref{thm:sbim_decomp} is now complete.

	\begin{remark}
		If we know the multiplicities of standard composition factors for our chosen Soergel bimodule, then we can immediately compute the direct sum decomposition of the specialisation. This is because the standard flag is preserved.
	\end{remark}
\begin{remark}
	The techniques applied here should also be valid in the setting of \cite{abe19}. A connection with the work of \cite{hazi} is a consequence.
\end{remark}

	\bibliography{references} 
	\bibliographystyle{alpha}	
	
\end{document}